\documentclass[reqno]{amsart}
\usepackage{csquotes}
\usepackage{hyperref}
\usepackage{amsmath}   
\usepackage{amssymb} 
\usepackage{mathrsfs} 
\usepackage{accents}
\usepackage{graphicx}

\usepackage{tabularx,colortbl}

\usepackage{tikz} 
\usetikzlibrary{calc} 
\usepackage{xcolor}  
\usetikzlibrary{arrows.meta} 
\usepackage{amsthm} 
\usepackage{float}
\usepackage{enumitem}
\usepackage[english]{babel}
\usepackage[font=footnotesize, labelfont=bf]{ caption}
\usepackage{ esint }
\usepackage{ dsfont }

\newcommand{\eg}{{\it e.g.}}
\newcommand{\ie}{{\it i.e.}}

\theoremstyle{plain}
\begingroup
\newtheorem{theorem}{Theorem}[section]
\newtheorem{lemma}[theorem]{Lemma}

\endgroup

\begingroup
\theoremstyle{definition}

\AtBeginEnvironment{definition}{%
  \pushQED{\qed}%
}
\AtEndEnvironment{definition}{\popQED\endexample}
\endgroup

\begingroup
\theoremstyle{remark}
\newtheorem{remark}[theorem]{Remark}
\AtBeginEnvironment{remark}{%
  \pushQED{\qed}%
}
\AtEndEnvironment{remark}{\popQED\endexample}
\newtheorem{example}[theorem]{Example}
\endgroup

\begingroup
\AtBeginEnvironment{example}{%
  \pushQED{\qed}%
}
\AtEndEnvironment{example}{\popQED\endexample}
\endgroup

\renewcommand{\tilde}{\widetilde}
\newcommand{\sm}{\setminus}

\renewcommand{\d}{ \mathrm{d}}

\newcommand{\Id}{\mathrm{Id}}

\newcommand{\Dom}{\mathrm{Dom}}

\numberwithin{equation}{section}
\newcommand{\N}{\mathbb{N}}
\newcommand{\Z}{\mathbb{Z}}
\newcommand{\R}{\mathbb{R}}

\newcommand{\de}{\partial} 

\newcommand{\A}{\mathscr{A}}

\newcommand{\x}{{\times}}

\renewcommand{\hat}{\widehat}

\usepackage{mathtools}
\mathtoolsset{showonlyrefs}  

\newcommand{\myequation}{\begin{equation}}
  \newcommand{\myendequation}{\end{equation}}
  \let\[\myequation
  \let\]\myendequation

  \newenvironment{acknowledgements}{%
  \begin{abstract}
}{%
  \end{abstract}
}

  \title[Frequency-dependent damping in the linear wave equation]{Frequency-dependent damping in the linear wave equation}

  \author[F. Maddalena]{F. Maddalena}
  \address[Francesco Maddalena]{\newline
   Dipartimento di Meccanica, Matematica e Management, Politecnico di Bari,
   Via E.~Orabona 4, I--70125 Bari, Italy.}
  \email[]{francesco.maddalena@poliba.it}
  
  \author[G. Orlando]{G. Orlando}
  \address[Gianluca Orlando]{\newline
   Dipartimento di Meccanica, Matematica e Management, Politecnico di Bari,
   Via E.~Orabona 4, I--70125 Bari, Italy.}
  \email[]{gianluca.orlando@poliba.it}


\usepackage[
backend=biber,
style=numeric-verb,
sorting=nty,
maxbibnames=99
]{biblatex}
\begin{document}

\begin{abstract}  
  We propose a model for frequency-dependent damping in the linear wave equation. After proving well-posedness of the problem, we study qualitative properties of the energy. In the one-dimensional case, we provide an explicit analysis for special choices of the damping operator. Finally, we show, in special cases, that solutions split into a dissipative and a conservative part.
\end{abstract}

\maketitle

\setcounter{tocdepth}{1}
\tableofcontents

\section{Introduction}

Sliding friction phenomena are nowadays recognized as the simultaneous occurrence of different physical mechanisms which escape the attempt to reduce them to a simple relation,  like a friction law. In fact, entire areas of scientific research, such as Tribology, are specialized in the study of these problems~\cite{Bro16, NosBhu08,NosMor14}.

From a mathematical (PDEs) perspective a reasonable  goal relies in detecting one or more key features which can enter into a rational model whose analysis reveals significant and new aspects of the phenomenon under examination. It was observed by various authors that in sliding friction the occurrence of local instabilities emerging at different scales play a decisive role in this physical manifestation~\cite{NosBhu08}.
Following this suggestion, we address here the simplest problem related to a linear wave equation in which a damping effect can  occur  only at certain spatial scales. This toy model is thought to investigate the effects of multiscale character in the wave propagation and the corresponding long-time behavior.

The reference model we have in mind is a variant of the linearly damped wave equation: 
\[ \label{eqintro:damped wave}
\begin{cases}
  \de_{tt} u(t,x) - \Delta u(t,x) + \de_t u(t,x) = 0 \, , & (t,x) \in (0,+ \infty) \times \Omega \, , \\
  u(t,x)= 0 \, , & (t,x) \in (0,+ \infty) \times \de \Omega \, , \\
  u(0,x) = u_0(x) \, , \quad \de_t u(0,x) = v_0(x) \, , & x \in \Omega \, .
\end{cases}
\]
It is well-known that the solution to this equation converges to zero (in $H^1$ norm) as $t \to + \infty$ exponentially fast~\cite{Ika97, ZhoSunLi18}.
The damping term $\de_t u$ in the equation is responsible for this decay. 
We are interested in studying a different scenario, in which the damping term does not act on the whole function $u$, but on suitable parts of it. 
In view of the above physical considerations, we could investigate the case in which damping regards a suitable frequency band of the solution. 
This is implemented by applying a linear operator $P$ to the damping term in the equation, which then reads 
\[ \label{eqintro:PDE}
\begin{cases}
  \de_{tt} u(t,x) - \Delta u(t,x) + P[\de_t u(t,\cdot)](x) = 0 \, , & (t,x) \in (0,+ \infty) \times \Omega \, , \\
  u(t,x)= 0 \, , & (t,x) \in (0,+ \infty) \times \de \Omega \, , \\
  u(0,x) = u_0(x) \, , \quad \de_t u(0,x) = v_0(x) \, , & x \in \Omega \, .
\end{cases}
\]
We illustrate the simplest model we have in mind, leaving more general assumptions on $P$ to the content of this work.
With the aim of capturing different behaviors at different spatial scales, one can consider the special case $\Omega = \R^d$ and the operator $P$ given by 
\[ 
\hat{P[v]}(t,\xi) = \chi_{\{|\xi| > 1\}}(\xi) \hat{v}(t,\xi) \, , \quad \xi \in \R^d \, ,
\]
where $\hat{v}(t,\xi)$ denotes the Fourier transform with respect to the spatial variable $x$ and $\chi_{\{|\xi| > 1\}}(\xi)$ is the characteristic function of the complement of the unit ball.
In this special case, due to the linearity of the equation, the different behavior of the solution at different scales becomes evident by inspecting it in the frequency variable. 
Indeed, in such a case, the problem is split into
\[ \label{eqintro:splitting}
\begin{cases}
\de_{tt} \hat u(t,\xi) + |\xi|^2 \hat u(t,\xi) + \de_t \hat u(t,\xi) = 0 \, , & t \in (0,+\infty) \, , \ |\xi| > 1 \, , \quad \text{(damped regime)}\\
\de_{tt} \hat u(t,\xi) + |\xi|^2 \hat u(t,\xi) = 0 \, , & t \in (0,+\infty) \, , \ |\xi| \leq 1 \, , \quad \text{(undamped regime)} \\
+ \text{ initial conditions.}& 
\end{cases}
\]
These different regimes are reflected in the long-time behavior of the solution. 
According to~\eqref{eqintro:splitting}, one expects that for high frequencies (\ie, small spatial scales) the solution behaves like the one to~\eqref{eqintro:damped wave} and converges exponentially fast to zero (dissipative regime). 
In contrast, for low frequencies (\ie, large spatial scales) the solution behaves like the one to the undamped wave equation and exhibits no decay (conservative regime).
 
We describe the results obtained in this paper concerning the problem~\eqref{eqintro:PDE}.
In Section~2, under suitable assumption on the operator $P$, we prove well-posedness of the problem~\eqref{eqintro:PDE} and we study qualitative decay properties of the energy. 
Section~3 is devoted to an explicit analysis for special choices of the operator $P$ in the one-dimensional case.
Finally, in Section~4, we show that, in special cases, solutions to~\eqref{eqintro:PDE} split into a dissipative and a conservative part.

Let us emphasize that the linear framework addressed in the present work constitutes the simplest setting in which we aim to investigate the role of the spatial scales in dissipative phenomena. 
We are aware that a correct scenario to grasp the very nature of this physics should include nonlinear terms. 
More complex behaviors may occur when damping is coupled with adhesion~\cite{BurKel78, CocDevMad21,CocFloLigMad17,CooBar73,LazMolRivSol22,LazMolSol22,MadPer08,DMLazNar16,CocDNMadOrlZua24}.
We plan to address these issues in future works.

\section{Description of the model and well-posedness} \label{sec:well-posedness}

In this section we introduce the model, we prove well-posedness of the problem~\eqref{eqintro:PDE}, and we study some qualitative properties of the energy.

\subsection{The model} Let $d \geq 1$ and assume that 
\begin{enumerate}[label=($\Omega$\arabic*)]
  \item \label{eq:Omega} $\Omega \subset \R^d$ is a bounded, connected, open set with boundary of class $C^2$.\footnote{These assumptions are enough to apply Poincar\'e's inequality in $H^1_0(\Omega)$ and the elliptic regularity theory up to the boundary.}
\end{enumerate}

We study the initial boundary value problem 
\[ \label{eq:PDE}
\begin{cases}
  \de_{tt} u(t,x) - \Delta u(t,x) + P[\de_t u(t,\cdot)](x) = 0 \, , & (t,x) \in (0,+ \infty) \times \Omega \, , \\
  u(t,x)= 0 \, , & (t,x) \in (0,+ \infty) \times \de \Omega \, , \\
  u(0,x) = u_0(x) \, , \quad \de_t u(0,x) = v_0(x) \, , & x \in \Omega \, ,
\end{cases}
\]
where $P \colon L^2(\Omega) \to L^2(\Omega)$ satisfies the following properties:
\begin{enumerate}[label=($P$\arabic*)]
  \item \label{eq:P1} $P$ is a bounded linear operator;
  \item \label{eq:P2} $P$ is monotone, \ie, $\langle P[u], u \rangle_{L^2(\Omega)} \geq 0$ for all $u \in L^2(\Omega)$.
\end{enumerate}

\subsection{The operator $P$}
Some comments on the properties~\ref{eq:P1}--\ref{eq:P2} of the operator $P$ are in order.

\begin{remark} \label{rem:damped_wave}
If $P = \Id_{L^2(\Omega)}$, we recover the classical wave equation with damping, \ie, 
equation~\eqref{eq:PDE} reduces to
\[ \label{eq:damped_wave}
\begin{cases}
  \de_{tt} u(t,x) - \Delta u(t,x) + \de_t u(t,x) = 0 \, , & (t,x) \in (0,+ \infty) \times \Omega \, , \\
  u(t,x)= 0 \, , & (t,x) \in (0,+ \infty) \times \de \Omega \, , \\
  u(0,x) = u_0(x) \, , \quad \de_t u(0,x) = v_0(x) \, , & x \in \Omega \, .
\end{cases}
\]
\end{remark}

To explain in which sense the term $P[\de_t u(t,\cdot)]$ may model a frequency-dependent damping, we point out the following example.

\begin{example} \label{ex:projection}
  If $P \colon L^2(\Omega) \to L^2(\Omega)$ is an orthogonal projection, then the properties~\ref{eq:P1}--\ref{eq:P2} are satisfied. 
  Indeed, orthogonal projections in an Hilbert space are characterized as linear self-adjoint idempotent operators. 
  Continuity follows from the fact that 
  \[
  \|P[v]\|_{L^2(\Omega)}^2 = \langle P[v], P[v] \rangle_{L^2(\Omega)} = \langle P^2[v], P[v] \rangle_{L^2(\Omega)} = \langle P[v], v \rangle_{L^2(\Omega)} \leq \|P[v]\|_{L^2(\Omega)} \|v\|_{L^2(\Omega)}  \, .
  \]
  From the inequality above, we also obtain that $P$ is monotone, since
  \[
  \langle P[v], v \rangle_{L^2(\Omega)} = \|P[v]\|_{L^2(\Omega)}^2 \geq 0 \, .
  \]
If an orthogonal projection is interpreted as a filter that selects some of the frequencies in the wave $u(t,\cdot)$, then the term $P[\de_t u(t,\cdot)]$ can be seen as a damping term acting on those selected frequencies.
For more details on this interpretation, see Section~\ref{sec:1d}.
\end{example}

\begin{remark} \label{rem:P_commutes_with_det}
  Since $P$ is acting only on the spatial variable, it commutes with the differentiation with respect to time.
  We explain this fact in a more rigorous way.
  Given $\varphi \in C^\infty_c(\R \times \Omega)$, we have that $t \mapsto P[\varphi]$ belongs to $C^\infty(\R;L^2(\Omega))$ and 
  \[ \label{eq:commutation_with_test}
    \de_t P[\varphi(t,\cdot)](x) = P[\de_t \varphi(t,\cdot)](x) \, , \quad \text{for all } t \in \R \, , \ x \in \Omega \, .
  \]
  This follows from the fact that, by linearity of $P$, for all $t \in \R$, $h \neq 0$, $x \in \Omega$ we have that
  \[
  \frac{P[\varphi(t+h,\cdot)](x) - P[\varphi(t,\cdot)](x)}{h} = P\Big[ \frac{\varphi(t+h,\cdot) - \varphi(t,\cdot)}{h} \Big](x) \, , 
  \]
  Exploiting boundedness of $P$ and passing to the limit as $h \to 0$, we deduce that $t \in \R \mapsto P[\varphi(t,\cdot)] \in L^2(\Omega)$ is differentiable and we obtain the claimed identity~\eqref{eq:commutation_with_test}.
  We iterate the argument for higher order derivatives to conclude that $t \mapsto P[\varphi(t,\cdot)]$ belongs to $C^\infty(\R;L^2(\Omega))$.
\end{remark}

\subsection{Well-posedness}

We recast~\eqref{eq:PDE} as an ODE in a Hilbert space.
By writing~\eqref{eq:PDE} as the first-order system
\[
\begin{cases}
  \de_t u(t,x) - v(t,x) = 0 \, , & (t,x) \in (0,+ \infty) \times \Omega \, , \\
  \de_t v(t,x) - \Delta u(t,x) + P[v(t,\cdot)](x) = 0 \, , & (t,x) \in (0,+ \infty) \times \Omega \, , \\
  u(t,x)= 0 \, , & (t,x) \in (0,+ \infty) \times \de \Omega \, , \\
  u(0,x) = u_0(x) \, , \quad v(0,x) = v_0(x) \, , & x \in \Omega \, ,
\end{cases}
\]
then~\eqref{eq:PDE} can be interpreted as the Cauchy problem for curves $U \colon [0,+\infty) \to H$
\[ \label{eq:ACP}
\begin{cases}
  \displaystyle \frac{\d U}{\d t}(t) + \A U(t) = 0\, , & t \in (0,+ \infty) \, , \\
  U(0) = U_0 \, ,
\end{cases}
\]
where 
\begin{itemize}
  \item $U(t) = (u(t), v(t))$ and $U_0 = (u_0, v_0)$;
  \item $H = H^1_0(\Omega) \times L^2(\Omega)$ is the Hilbert space endowed with the scalar product (we are exploiting Poincar\'e 's inequality)
  \[
  \langle (u,v), (w,z) \rangle_H = \langle \nabla u, \nabla w \rangle_{L^2(\Omega)} + \langle v, z \rangle_{L^2(\Omega)} \, ;
  \]
  \item the linear operator $\A \colon \Dom(\A) \to H$ is defined by
  \[
  \A(u,v) = (-v, - \Delta u + P[v]) \, , \quad (u,v) \in \Dom(\A) \, ;
  \] 
  \item the domain $\Dom(\A)$ is given by
  \[
  \Dom(\A) = \{ (u,v) \in H \, : \, u \in H^2(\Omega) \cap H^1_0(\Omega) \, , \, v \in H^1_0(\Omega) \} \, .
  \]
\end{itemize}

We start with the key result that guarantees the well-posedness of the problem, relying on the theory of maximal monotone operators~\cite{Har87}.

\begin{lemma}
  Assume~\ref{eq:Omega} and~\ref{eq:P1}--\ref{eq:P2}. 
  Then the operator $\A \colon \Dom(\A) \to H$ is maximal monotone.
\end{lemma}
\begin{proof}
  We have to show the following properties:
  \begin{enumerate}[label={\itshape (\roman*)}]
    \item \label{item:monotone} $\A$ is monotone, \ie, for all $(u,v), (w,z) \in \Dom(\A)$ we have
    \[
    \langle \A(u,v) - \A(w,z), (u-w,v-z) \rangle_H \geq 0 \, ;  
    \]
    \item \label{item:maximal} $\A$ is maximal monotone, which is equivalent to say that $\A + \lambda \Id_H$ is surjective for some (and hence all) $\lambda > 0$.
  \end{enumerate}

  Let us prove~\ref{item:monotone}. 
  By linearity of $\A$, it is enough to show that $\A$ satisfies
  \[
  \langle \A(u,v), (u,v) \rangle_H \geq 0 \, , \quad (u,v) \in \Dom(\A) \, .
  \]
  Let $(u,v) \in \Dom(\A)$. 
  Integrating by parts and by~\ref{eq:P2}, we have that
  \[
  \begin{split}
    \langle \A(u,v), (u,v) \rangle_H & = \langle (-v, - \Delta u + P[v] + u), (u,v) \rangle_H \\
    & = - \langle \nabla v, \nabla u \rangle_{L^2(\Omega)} + \langle -\Delta u , v \rangle_{L^2(\Omega)} + \langle P[v], v \rangle_{L^2(\Omega)} \\
    & = - \langle \nabla v, \nabla u \rangle_{L^2(\Omega)} + \langle \nabla u, \nabla v \rangle_{L^2(\Omega)} + \langle P[v], v \rangle_{L^2(\Omega)}  = \langle P[v], v \rangle_{L^2(\Omega)} \geq 0 \, .
  \end{split}
  \]

  Let us prove~\ref{item:maximal}.
  We show that $\A + \Id_H$ is surjective.
  Let $(f,g) \in H$.
  We show that it is possible to find $(u,v) \in \Dom(\A)$ such that
  \[
  \begin{cases}
    -v + u = f \, , \\
    - \Delta u + P[v] + v = g \, .
  \end{cases}
  \]
  Indeed, the first equation gives $v = -f + u$.
  We substitute in the second equation to obtain that 
  \[
  - \Delta u + P[-f + u] + u - f = g \implies - \Delta u + P[u] + u  = f + P[f] + g \, .
  \] 
  The problem then reduces to finding $u \in H^2(\Omega)$ such that
  \[ \label{eq:surjectivity_problem}
  \begin{cases}
    - \Delta u + P[u] + u  = h \, , & \text{in } \Omega \, , \\
    u = 0 \, , & \text{on } \de \Omega \, ,
  \end{cases}
  \]
  where $h = f + P[f] + g \in L^2(\Omega)$.
  This problem has a unique solution $u \in H^1_0(\Omega)$ by the Lax-Milgram theorem.
  Indeed, the bilinear form
  \[
  b(u,u') = \langle \nabla u, \nabla u' \rangle_{L^2(\Omega)} + \langle u, u' \rangle_{L^2(\Omega)} + \langle P[u], u' \rangle_{L^2(\Omega)} \, , \quad u,u' \in H^1_0(\Omega) \, ,
  \]
  satisfies
  \[
  |b(u,u')| \leq C \|u\|_{H^1_0(\Omega)} \|u'\|_{H^1_0(\Omega)} \, , \quad u,u' \in H^1_0(\Omega) \, ,
  \]
  by the continuity of $P$ and Poincar\'e's inequality. 
  Moreover, it is coercive, since
  \[
  b(u,u) = \|\nabla u\|_{L^2(\Omega)}^2 + \|u\|_{L^2(\Omega)}^2 + \langle P[u], u \rangle_{L^2(\Omega)} \geq \|\nabla u\|_{L^2(\Omega)}^2 + \|u\|_{L^2(\Omega)}^2 \geq c \|u\|_{H^1_0(\Omega)}^2 \, ,
  \]
  by~\ref{eq:P2} and Poincar\'e's inequality.

  Then we read~\eqref{eq:surjectivity_problem} as
  \[
    \begin{cases}
      - \Delta u = h - P[u] - u\, , & \text{in } \Omega \, , \\
      u = 0 \, , & \text{on } \de \Omega \, ,
    \end{cases}
  \]
  with the right-hand side $h - P[u] - u \in L^2(\Omega)$.
  By the elliptic regularity theory, we conclude that $u \in H^2(\Omega) \cap H^1_0(\Omega)$.
  This concludes the proof.
\end{proof}

Thanks to the maximal monotonicity of $\A$, we can apply the theory of maximal monotone operators to guarantee the well-posedness of the Cauchy problem~\eqref{eq:ACP}.
We have the following classical result.

\begin{theorem} \label{thm:well-posedness}
  Assume~\ref{eq:Omega} and~\ref{eq:P1}--\ref{eq:P2}. 
  Then the operator $\A$ generates a continuous semigroup $\{S(t)\}_{t \geq 0}$ on $H$, \ie, for all $t \geq 0$ the operator $S(t) \colon H \to H$ is linear with $\|S(t)\| \leq 1$, $S(0) = \Id_H$, $S(t+s) = S(t)S(s)$ for all $t,s \geq 0$, and $\lim_{t \to 0} \|S(t)U_0 - U_0\|_H = 0$ for all $U_0 \in H$.
  Moreover, if $U_0 \in \Dom(\A)$, then $U(t) = S(t)U_0$ belongs to $C^1([0,+\infty);H) \cap C([0,+\infty); \Dom(\A))$ and is the unique solution to~\eqref{eq:ACP}. 
\end{theorem} 
\begin{proof}
  The result follows from the theory of maximal monotone operators, see, \eg, \cite[Theorem I.2.2.1 \& Remark I.2.2.3]{Har87}.
\end{proof}

\subsection{Energy-dissipation balance} 
In this subsection we show that the energy of the solution to~\eqref{eq:PDE} is non-increasing in time and we provide an explicit expression for the energy dissipated by the damping term. 
Given $(u,v) \in H^1_0(\Omega) \times L^2(\Omega)$, we define the energy functional by
\[
E(u,v) = \frac{1}{2} \int_\Omega |v(x)|^2 \, \d x + \frac{1}{2} \int_\Omega |\nabla u(x)|^2 \, \d x \, .
\]

\begin{theorem}
  Assume~\ref{eq:Omega} and~\ref{eq:P1}--\ref{eq:P2}. 
  Let $U_0 = (u_0, v_0) \in H^1_0(\Omega) \times L^2(\Omega)$.
  Let $U(t) = (u(t), v(t))$ be the unique solution to~\eqref{eq:ACP} with initial datum $U_0$ provided by Theorem~\ref{thm:well-posedness}.
  Then the following energy-dissipation balance holds:
\[ \label{eq:energy_dissipation_balance}
E(u(t), v(t)) + \int_0^t \langle P[v(s)], v(s) \rangle_{L^2(\Omega)} \, \d s = E(u_0, v_0) \, , \quad t \in [0,+\infty) \, .
\]
In particular, the energy of the solution is non-increasing in time, \ie, 
\[
E(u(s), v(s)) \geq E(u(t), v(t)) \, , \quad 0 \leq s \leq t \, .
\] 
\end{theorem} 
\begin{proof}
  Let us start by proving~\eqref{eq:energy_dissipation_balance} in the case $U_0 = (u_0, v_0) \in \Dom(\A)$.
  In this case, by Theorem~\ref{thm:well-posedness} the solution $U(t) = (u(t), v(t))$ belongs to $C^1([0,+\infty);H) \cap C([0,+\infty); \Dom(\A))$.
  It follows that $t \mapsto E(u(t), v(t))$ is differentiable. 
  Differentiating the energy functional, substituting the equation in~\eqref{eq:PDE}, and integrating by parts, we obtain
  \[
  \begin{split}
    \frac{\d}{\d t} E(u(t), v(t)) & = \frac{\d}{\d t} E(u(t), \de_t u (t)) = \langle \de_t u(t), \de_{tt} u(t) \rangle_{L^2(\Omega)} + \langle \nabla u(t), \nabla \de_t u(t) \rangle_{L^2(\Omega)} \\
    & = \langle \de_t u(t), \Delta u(t) - P[\de_t u(t)] \rangle_{L^2(\Omega)} - \langle \Delta u(t), \de_t u(t) \rangle_{L^2(\Omega)} \\
    & = - \langle P[\de_t u(t)], \de_t u(t) \rangle_{L^2(\Omega)} \, .
  \end{split}
  \]
  Integrating in time, we obtain~\eqref{eq:energy_dissipation_balance}.

  Assuming that $U_0 = (u_0, v_0) \in H^1_0(\Omega) \times L^2(\Omega)$, we approximate it by a sequence $U_0^n = (u_0^n, v_0^n) \in \Dom(\A)$ such that $U_0^n \to U_0$ in $H$ as $n \to + \infty$.
  By the continuity of the semigroup obtained in Theorem~\ref{thm:well-posedness}, we have that $U^n(t) = (u^n(t), v^n(t)) = S(t)U_0^n \to U(t) = (u(t),v(t)) = S(t)U_0$ in $H$ as $n \to +\infty$, for all $t \geq 0$.
  The energy-dissipation balance~\eqref{eq:energy_dissipation_balance} holds for $U_0^n$ and $U^n(t)$, \ie,
  \[
  E(u^n(t), v^n(t)) + \int_0^t \langle P[v^n(s)], v^n(s) \rangle_{L^2(\Omega)} \, \d s = E(u_0^n, v_0^n) \, , \quad t \in [0,+\infty) \, .
  \]
  By the continuity of the energy functional with respect to the norm of $H$, we pass to the limit as $n \to + \infty$ in the above equation to obtain~\eqref{eq:energy_dissipation_balance} for $U_0$ and $U(t)$.

  Finally, to show that the energy is non-increasing in time, we observe that 
  \[
  E(u(t), v(t)) + \int_s^t \langle P[v(\tau)], v(\tau) \rangle_{L^2(\Omega)} \, \d \tau = E(u(s), v(s)) \, , \quad 0 \leq s \leq t \, ,
  \]
  and, by the monotonicity of $P$, we have that $\langle P[v(\tau)], v(\tau) \rangle_{L^2(\Omega)} \geq 0$ for all $\tau \geq 0$.
  This concludes the proof.
\end{proof}

\begin{remark} \label{rem:exponential_decay_damped_wave}
  If $P = \Id_{L^2(\Omega)}$, the equation~\eqref{eq:PDE} reduces to the damped wave equation~\eqref{eq:damped_wave}, see Remark~\ref{rem:damped_wave}.
  A classical result states that the energy of the solution to the damped wave equation decays exponentially fast to zero.
  We recall the proof of this result, which relies on the perturbed energy functional
  \[
  E_\lambda(u,v) = \frac{1}{2} \int_\Omega |v(x)|^2 \, \d x + \frac{1}{2} \int_\Omega |\nabla u(x)|^2 \, \d x + \lambda \int_\Omega u(x)v(x) \, \d x \, ,
  \]
  where $\lambda > 0$ is a parameter to be chosen.
  Assuming that the initial datum $U_0 = (u_0, v_0)$ belongs to $\Dom(\A)$, we have that the solution $U(t) = (u(t), v(t))$ provided by Theorem~\ref{thm:well-posedness} belongs to $C^1([0,+\infty);H) \cap C([0,+\infty); \Dom(\A))$. 
  We can differentiate the perturbed energy functional and substitute the equation to obtain that 
  \[
  \begin{split}
    \frac{\d}{\d t} E_\lambda(u(t), v(t)) & =  -  \|v(t)\|_{L^2(\Omega)}^2 + \lambda \|v(t)\|_{L^2(\Omega)}^2 + \lambda \langle u(t), \Delta u(t) \rangle_{L^2(\Omega)} - \lambda \langle u(t), v(t) \rangle_{L^2(\Omega)} \\
    & = - (1 - \lambda) \|v(t)\|_{L^2(\Omega)}^2 - \lambda \| \nabla u(t) \|^2_{L^2(\Omega)} - \lambda \langle u(t), v(t) \rangle_{L^2(\Omega)} \, .
  \end{split}
  \]
  Choosing $0 < \lambda < \frac{1}{2}$, we obtain that
  \[
  \frac{\d}{\d t} E_{\lambda}(u(t), v(t)) \leq - C_\lambda E_{\lambda}(u(t), v(t)) \, , 
  \]
  for a suitable $C_\lambda > 0$, which implies that 
  \[
  E_{\lambda}(u(t), v(t)) \leq E_{\lambda}(u_0, v_0) e^{- \gamma t} \, ,
  \]
  with $\gamma > 0$ depending on $\lambda$.
  Then we estimate, by Poincar\'e's inequality, that
  \[
  \big| \langle u(t), v(t) \rangle_{L^2(\Omega)} \big| \leq \frac{1}{2} \|u(t)\|_{L^2(\Omega)}^2 + \frac{1}{2} \|v(t)\|_{L^2(\Omega)}^2 \leq C \|\nabla u(t)\|_{L^2(\Omega)} + \frac{1}{2} \|v(t)\|_{L^2(\Omega)}^2 \, ,
  \]
  where $C > 0$ depends on the Poincar\'e constant of $\Omega$.
  By choosing $\lambda$ suitably small, we obtain the estimate 
  \[
  \frac{1}{C} E_{\lambda}(u(t), v(t)) \leq E(u(t), v(t)) \leq C E_{\lambda}(u(t), v(t)) \, ,
  \]
  which implies that 
  \[
  E(u(t), v(t)) \leq C E(u_0, v_0) e^{- \gamma t} \, .
  \]
  The estimate above is extended to the case $U_0 = (u_0, v_0) \in H^1_0(\Omega) \times L^2(\Omega)$ by approximation.
  This energy decay implies that the solution to the damped wave equation converges exponentially fast to zero in the energy space.
\end{remark}

\section{Explicit solution in the one-dimensional case for filtered damping} \label{sec:1d}

In this section we provide an explicit solution to the problem~\eqref{eq:PDE} in the one-dimensional case, \ie, $\Omega = (0,L)$, assuming a specific structure for the operator $P$.
This will allow us to understand the asymptotic behavior of the solution in some specific cases.

\subsubsection*{Formulation of the problem in one dimension} We aim to solve the problem
\[ \label{eq:PDE 1d}
\begin{cases}
  \de_{tt}u  - \de_{xx}u + P[\de_t  u(t,\cdot)](x) = 0 \, , & (t,x) \in (0,+ \infty) \times (0,L) \, , \\
  u(t,0) = u(t,L) = 0 \, , & t \in [0,+ \infty) \, , \\
  u(0,x) = u_0(x) \, , \quad u_t(0,x) = v_0(x) & x \in (0,L) \, ,
\end{cases}
\]
where $u_0 \in H^1_0(0,L)$ and $v_0 \in L^2(0,L)$. 

\subsubsection*{Fourier series notation} To solve explicitly the problem, we will regard solutions as $L$-periodic functions and use Fourier series. We write
\[
u_0(x) = \sum_{k\in \Z} \hat{u}_k(0) e^{2 \pi i k x / L} \, , \quad v_0(x) = \sum_{k\in \Z} \hat{v}_k(0) e^{2 \pi i k x / L} \, , \quad u(t,x) = \sum_{k\in \Z} \hat{u}_k(t) e^{2 \pi i k x / L} \, ,
\]
where the Fourier coefficients are given by
\[
\hat{u}_k(t) = \frac{1}{L} \int_0^L u(t,x) e^{-2 \pi i k x / L} \, \d x \, , \quad k \in \Z \, ,
\]
and analogously for $u_0$ and $v_0$.
Since $u_0$ and $v_0$ are real-valued, we have that 
\[
\hat{u}_{-k}(0) = \overline{\hat{u}_k(0)} \, , \quad \hat{v}_{-k}(0) = \overline{\hat{v}_k(0)}  \, .
\] 

\subsubsection*{Assumption on the operator $P$} We analyze a specific form for the damping term $P[\de_t u(t,\cdot)]$, assuming that $P$ acts as a frequency filter.
More precisely, we assume that 
\begin{enumerate}[label=($P$\arabic*)]
  \setcounter{enumi}{2}
  \item \label{eq:P3} there exists a bounded sequence $(\hat \phi_k)_{k \in \Z}$ with $\hat \phi_k \geq 0$ for all $k \in \Z$ such that
\[ 
(\hat{P[v]})_k = \hat{\phi}_k \hat{v}_k \,. 
\]
\end{enumerate}

\begin{remark}
  If $P$ satisfies~\ref{eq:P3}, by Parseval's identity, we have that
  \[
  \|P[v]\|_{L^2(0,L)}^2 = \sum_{k \in \Z} |\hat \phi_k \hat v_k|^2 \leq \sup_{k \in \Z} |\hat \phi_k|^2 \|v\|_{L^2(0,L)}^2 \, ,
  \]
  whence~\ref{eq:P1} holds.
  Moreover, 
  \[
  \langle P[v], v \rangle_{L^2(0,L)} = \sum_{k \in \Z} \hat \phi_k \hat v_k^2 \geq 0 \, ,
  \]
  \ie, \ref{eq:P2} holds.
\end{remark}

\begin{example}  
  If $P$ is obtained via a convolution with an $L^1$ function, \ie, 
  \[
  P[v](x) = \phi * v(x) = \frac{1}{L}\int_0^L \phi(x-y) v(y) \, \d y \, ,
  \]
  with $\phi \in L^1(0,L)$ extended periodically, then~\ref{eq:P3} is satisfied.
  Indeed, we have that 
  \[
    (\hat{P[v]})_k = (\hat{ \phi * v})_k = \hat{\phi}_k \hat{v}_k \, .
  \]
\end{example}

\subsubsection*{Solving the PDE with Fourier series} Taking the Fourier series in the equation in~\eqref{eq:PDE 1d} we obtain that the Fourier coefficients $\hat{u}_k(t)$ satisfy the following system of ordinary differential equations 
\[ \label{eq:ODE_hat_u_k}
  \de_{tt} \hat{u}_k(t) + \big(\tfrac{2 \pi k}{L} \big)^2 \hat{u}_k(t) + \hat \phi_k \de_t \hat{u}_k(t) = 0   \, , \quad k \in \Z \, .
\]
This can be solved explicitly. 
The characteristic equation is
\[
\lambda^2 + \hat \phi_k \lambda + \big(\tfrac{2 \pi k}{L} \big)^2 = 0 \, ,
\]
and its solutions are, possibly counted with multiplicity,
\[ \label{eq:lambda_k_pm}
\lambda_k^\pm = \frac{- \hat \phi_k \pm \sqrt{\hat \phi_k^2 - \left( \frac{4 \pi k}{L} \right)^2}}{2} \,.
\]

{\itshape Case 1:} $\hat \phi_k^2 - \left( \frac{4 \pi k}{L} \right)^2 > 0$. The roots obtained in~\eqref{eq:lambda_k_pm} are real and distinct. 
The general solution of~\eqref{eq:ODE_hat_u_k} is
\[ \label{eq:general_solution_1d_delta_positive}
\hat{u}_k(t) =  A_k e^{t \lambda_k^+} + B_k e^{t \lambda_k^-}  \, ,
\]
where $A_k$ and $B_k$ are complex constants that depend on the initial data.
To find them, we impose 
\[
\begin{cases}
A_k + B_k = \hat u_k(0) \, , \\
A_k \lambda_k^+ + B_k \lambda_k^- = \hat v_k(0) \, .
\end{cases}
\]
Solving this system we obtain
\[
  A_k = - \frac{\hat u_k(0) \lambda_k^- - \hat v_k(0)}{\lambda_k^+ - \lambda_k^-} \, , \quad B_k = \frac{\hat u_k(0) \lambda_k^+ - \hat v_k(0)}{\lambda_k^+ - \lambda_k^-} \, .
\]
Substituting in~\eqref{eq:general_solution_1d_delta_positive} we obtain
\[ 
\begin{split}
  \hat u_k(t) & = - \frac{\hat u_k(0) \lambda_k^- - \hat v_k(0)}{\lambda_k^+ - \lambda_k^-} e^{t \lambda_k^+} + \frac{\hat u_k(0) \lambda_k^+ - \hat v_k(0)}{\lambda_k^+ - \lambda_k^-} e^{t \lambda_k^-} \\
  & = \hat u_k(0) \frac{- \lambda_k^- e^{t \lambda_k^+} + \lambda_k^+ e^{t \lambda_k^-} }{\lambda_k^+ - \lambda_k^-} + \hat v_k(0)\frac{e^{t \lambda_k^+} - e^{t \lambda_k^-}}{\lambda_k^+ - \lambda_k^-}  \\
  & = \hat u_k(0) \frac{ \big( \frac{1}{2} \hat \phi_k + \frac{1}{2} \sqrt{\hat \phi_k^2 - \big(\tfrac{4\pi k}{L}\big)^2} \big) e^{t \lambda_k^+} + \big( - \frac{1}{2} \hat \phi_k + \frac{1}{2} \sqrt{\hat \phi_k^2 - \big(\tfrac{4\pi k}{L}\big)^2} \big) e^{t \lambda_k^-} }{\sqrt{\hat \phi_k^2 - \big(\tfrac{4\pi k}{L}\big)^2}} \\
  & \quad + \hat v_k(0)\frac{e^{t \lambda_k^+} - e^{t \lambda_k^-}}{\sqrt{\hat \phi_k^2 - \big(\tfrac{4\pi k}{L}\big)^2}} \\
  & = \hat u_k(0) \frac{e^{t \lambda_k^+} + e^{t \lambda_k^-}}{2} + \Big(  \hat \phi_k  \hat u_k(0) + 2\hat v_k(0) \Big)\frac{1}{\sqrt{\hat \phi_k^2 - \big(\tfrac{4\pi k}{L}\big)^2}} \frac{e^{t \lambda_k^+} - e^{t \lambda_k^-}}{2} \, .
\end{split} 
\]
Writing the explicit expression of $\lambda_k^\pm$ we obtain
\[
\begin{split}
  \hat u_k(t) = e^{- t  \frac{\hat \phi_k}{2}} \Big[ & \cosh \Big( \tfrac{t}{2}\sqrt{\hat \phi_k^2 - \big(\tfrac{4\pi k}{L}\big)^2}  \Big) \hat u_k(0)  \\
  & + \frac{1}{\frac{1}{2}\sqrt{\hat \phi_k^2 - \big(\tfrac{4\pi k}{L}\big)^2}}   \sinh \Big( \tfrac{t}{2}\sqrt{\hat \phi_k^2 - \big(\tfrac{4\pi k}{L}\big)^2}  \Big) \Big(\tfrac{\hat \phi_k}{2}  \hat u_k(0) + \hat v_k(0)\Big) \Big]\, .
\end{split}
\]

{\itshape Case 2:} $\hat \phi_k^2 - \left( \frac{4 \pi k}{L} \right)^2 < 0$. The roots obtained in~\eqref{eq:lambda_k_pm} are complex.
The algebra to obtain the solution is the same as in the previous case. We obtain
\[
  \begin{split}
    \hat u_k(t) = e^{- t  \frac{\hat \phi_k}{2}} \Big[ & \cos \Big( \tfrac{t}{2}\sqrt{\big(\tfrac{4\pi k}{L}\big)^2 - \hat \phi_k^2 }  \Big) \hat u_k(0)  \\
    & + \frac{1}{\frac{1}{2}\sqrt{\big(\tfrac{4\pi k}{L}\big)^2 - \hat \phi_k^2 }}   \sin \Big( \tfrac{t}{2}\sqrt{\big(\tfrac{4\pi k}{L}\big)^2 - \hat \phi_k^2 }  \Big) \Big(\tfrac{\hat \phi_k}{2}  \hat u_k(0) + \hat v_k(0)\Big) \Big]\, .
  \end{split}
  \]
  
{\itshape Case 3:} $\hat \phi_k^2 - \left( \frac{4 \pi k}{L} \right)^2 = 0$. The roots obtained in~\eqref{eq:lambda_k_pm} are real and coincide.
This means that the general solution of~\eqref{eq:ODE_hat_u_k} is
\[ \label{eq:general_solution_1d_delta_zero}
\hat{u}_k(t) =  A_k e^{-t \frac{\hat \phi_k}{2}} + B_k t e^{-t \frac{\hat \phi_k}{2}}  \, .
\]
To find the constants $A_k$ and $B_k$ we impose
\[
\begin{cases}
A_k = \hat u_k(0) \, , \\
A_k \big( - \frac{\hat \phi_k}{2} \big) + B_k  = \hat v_k(0) \, .
\end{cases} 
\]
Solving this system we obtain
\[
A_k = \hat u_k(0) \, , \quad B_k = \tfrac{\hat \phi_k}{2} \hat u_k(0) + \hat v_k(0)  \, .
\]
Substituting in~\eqref{eq:general_solution_1d_delta_zero} we obtain
\[
\hat u_k(t) = e^{-t \frac{\hat \phi_k}{2}} \Big[ \hat u_k(0)  + \Big(\tfrac{\hat \phi_k}{2} \hat u_k(0) + \hat v_k(0) \Big) t   \Big]  \, .
\]

\subsubsection*{Result in one dimension} In the previous subsection we have obtained the explicit solution to the problem~\eqref{eq:PDE 1d} in the one-dimensional case.
The result is summarized in the following theorem.

\begin{theorem} \label{thm:explicit_solution_1d}
  Assume that $\Omega = (0,L)$ and that the operator $P$ satisfies~\ref{eq:P3}.
  Let $u_0 \in H^1_0(0,L)$ and $v_0 \in L^2(0,L)$.
  Let $u(t,x)$ be the unique solution to~\eqref{eq:PDE 1d} with initial data $u_0$ and $v_0$.
  Then the solution $u(t,x)$ is given by
  \[
    u(t,x) = \sum_{k \in \Z} \hat u_k(t) e^{2 \pi i k x / L} \, ,
  \]
  where the Fourier coefficients $\hat u_k(t)$ are given by the following expressions:
  \begin{itemize}
    \item if $\hat \phi_k^2 - \left( \frac{4 \pi k}{L} \right)^2 > 0$, then 
  \[ \label{eq:explicit_solution_1d_delta_positive}
  \begin{split}
    \hat u_k(t) = e^{- t  \frac{\hat \phi_k}{2}} \Big[ & \cosh \Big( \tfrac{t}{2}\sqrt{\hat \phi_k^2 - \big(\tfrac{4\pi k}{L}\big)^2}  \Big) \hat u_k(0)  \\
  & + \frac{1}{\frac{1}{2}\sqrt{\hat \phi_k^2 - \big(\tfrac{4\pi k}{L}\big)^2 }}   \sinh \Big( \tfrac{t}{2}\sqrt{\hat \phi_k^2 - \big(\tfrac{4\pi k}{L}\big)^2}  \Big) \Big(\tfrac{\hat \phi_k}{2}  \hat u_k(0) + \hat v_k(0)\Big) \Big]\, ;
  \end{split}
  \]
  \item if $\hat \phi_k^2 - \left( \frac{4 \pi k}{L} \right)^2 < 0$, then
  \[ \label{eq:explicit_solution_1d_delta_negative}
  \begin{split}
  \hat u_k(t) = e^{- t  \frac{\hat \phi_k}{2}} \Big[ & \cos \Big( \tfrac{t}{2}\sqrt{\big(\tfrac{4\pi k}{L}\big)^2 - \hat \phi_k^2 }  \Big) \hat u_k(0)  \\
    & + \frac{1}{\frac{1}{2}\sqrt{\big(\tfrac{4\pi k}{L}\big)^2- \hat \phi_k^2}}   \sin \Big( \tfrac{t}{2}\sqrt{\big(\tfrac{4\pi k}{L}\big)^2 - \hat \phi_k^2 }  \Big) \Big(\tfrac{\hat \phi_k}{2}  \hat u_k(0) + \hat v_k(0)\Big) \Big]\, ;
  \end{split}
  \] 
  \item if $\hat \phi_k^2 - \left( \frac{4 \pi k}{L} \right)^2 = 0$, then
  \[ \label{eq:explicit_solution_1d_delta_zero}
  \hat u_k(t) = e^{-t \frac{\hat \phi_k}{2}} \Big[ \hat u_k(0)  + \Big(\tfrac{\hat \phi_k}{2} \hat u_k(0) + \hat v_k(0) \Big) t   \Big]  \, .
  \]
\end{itemize}
\end{theorem}

\subsubsection*{Examples} We provide some examples of the explicit solution in the one-dimensional case.
In the next two examples we show that the solution converges exponentially fast to the solution projected on the null space of the frequency filter $P$.

\begin{example} \label{ex:filter_1d}
  Let $(\hat \phi_k)_{k\in}$ be such that 
  \[
  \hat \phi_k = \begin{cases}
    1 \, , & \text{if } |k| \geq k_0 \, , \\ 
    0 \, , & \text{if } |k| < k_0 \, ,
  \end{cases}
  \]
  for some $k_0 \in \N$.
  This means that the operator $P \colon L^2(0,L) \to L^2(0,L)$ is an orthogonal projection on Fourier modes with high frequencies. 
  Indeed, it is idempotent, \ie, $P^2 = P$, and self-adjoint.

  Assume that $1 - 4 \left( \frac{2 \pi k}{L} \right)^2 < 0$ for $|k| \geq k_0$.
  In this case, the solution is given by~\eqref{eq:explicit_solution_1d_delta_negative}.
  In particular, 
  \[
  \hat u_k(t) = \cos \Big( t   \tfrac{2\pi|k|}{L}  \Big) \hat u_k(0) + \frac{1}{\tfrac{2\pi|k|}{L}}   \sin \Big( t \tfrac{2\pi|k|}{L} \Big)   \hat v_k(0)  \, ,
  \]
  for $|k| < k_0$, and
  \[ \label{eq:example_high_frequencies}
  \hat u_k(t) = e^{- \frac{t}{2}} \Big[  \cos \Big( \tfrac{t}{2}\sqrt{\big(\tfrac{4\pi k}{L}\big)^2 - 1 }  \Big) \hat u_k(0) + \frac{2}{\sqrt{\big(\tfrac{4\pi k}{L}\big)^2 - 1}}   \sin \Big( \tfrac{t}{2}\sqrt{\big(\tfrac{4\pi k}{L}\big)^2 - 1}  \Big) \Big(\tfrac{1}{2}  \hat u_k(0) + \hat v_k(0)\Big) \Big]\, ,
  \]
  for $|k| \geq k_0$.

  From the explicit solutions we can deduce the asymptotic behavior of the solution.
  First of all, we observe that the null space of the linear operator $P \colon L^2(0,L) \to L^2(0,L)$ is given by
  \[
  N(P) = \Big\{v = \sum_{|k| < k_0} \hat v_k e^{2\pi i k x/L} \in L^2(0,L) \Big\} \, .
  \]
  We define $Q = \Id - P \colon L^2(0,L) \to L^2(0,L)$, which is the orthogonal projection on the null space of $P$, \ie, the space of Fourier modes with low frequencies.
  We observe that 
  \[
  \hat{Q[u(t,\cdot)]}_k = \begin{cases}
    \cos \Big( t   \tfrac{2\pi|k|}{L}  \Big) \hat u_k(0) + \frac{1}{\tfrac{2\pi|k|}{L}}   \sin \Big( t \tfrac{2\pi|k|}{L} \Big)   \hat v_k(0) \, , & \text{if } |k| < k_0 \, , \\
    0 \, , & \text{if } |k| \geq k_0 \, .
  \end{cases}
  \]
  This implies that $\|u(t,\cdot) - Q[u(t,\cdot)]\|^2_{H^1(0,L)}$ can be estimated, using Parseval's identity, simply in terms of the Fourier coefficients in~\eqref{eq:example_high_frequencies}, giving 
  \[
  \|u(t,\cdot) - Q[u(t,\cdot)]\|_{H^1(0,L)} \leq C e^{-\frac{t}{2}} \Big( \|u_0\|_{H^1(0,L)} + \|v_0\|_{L^2(0,L)} \Big) \, .
  \]
  Moreover, by linearity, $Q[u(t,\cdot)]$ is the solution to the problem~\eqref{eq:PDE 1d} with initial data $Q[u_0]$ and $Q[v_0]$, \ie, the initial data projected on the null space of $P$.
  We have shown that the solution $u(t,x)$ to the problem~\eqref{eq:PDE 1d} converges exponentially fast to the solution to the problem with initial data projected on the null space of $P$. 
  See also Figure~\ref{fig:example_high_frequencies} for a numerical simulation.
  \begin{figure}[H]
    \begin{minipage}{\textwidth}
      \centering
      \includegraphics[width=0.45\textwidth]{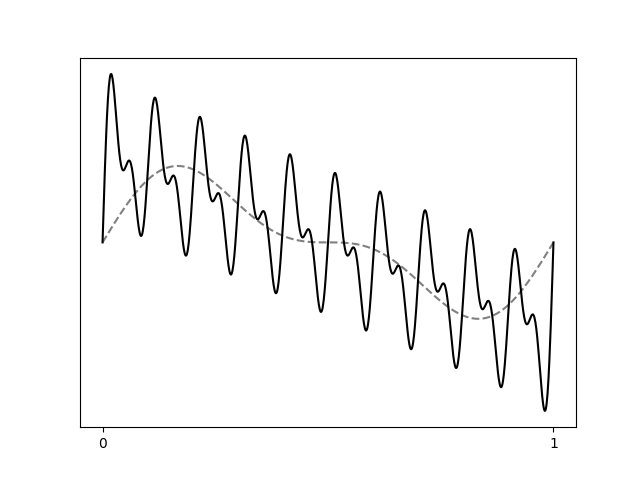} \quad 
      \includegraphics[width=0.45\textwidth]{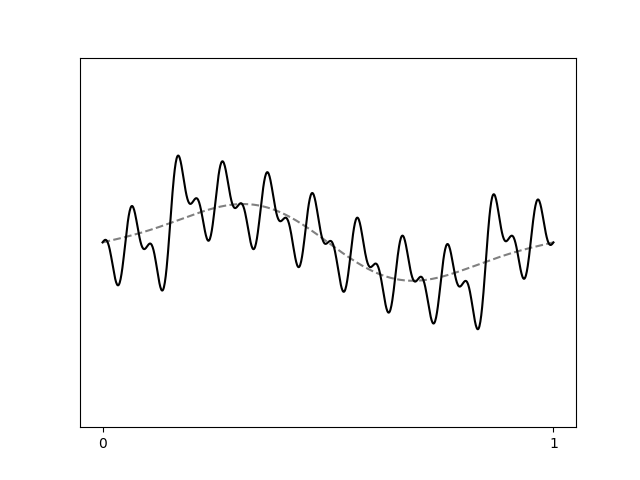} \\ 
      \includegraphics[width=0.45\textwidth]{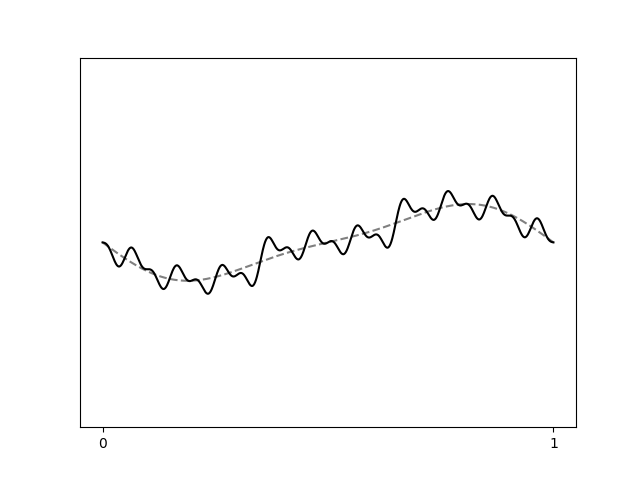} \quad 
      \includegraphics[width=0.45\textwidth]{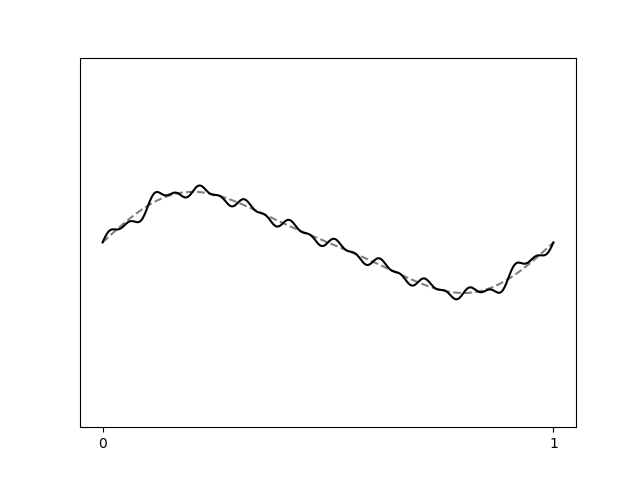}
    \end{minipage}
    \caption{Numerical simulation of the solution to the problem~\eqref{eq:PDE 1d} with $L = 1$, $u_0$ built on modes with frequencies ranging from $0$ to $20$, $v_0 = 0$, and $\phi$ such that $\hat \phi_k = 1$ for $|k| < 3$ and $\hat \phi_k = 0$ for $|k| \geq 3$. The figure shows on top left the initial condition $u_0$ in solid black and its projection $Q[u_0]$ on the null space of $P$ in dashed black. The other frames show the solution $u(t,\cdot)$ at different times. 
     The solution converges exponentially fast to the solution built only on the modes with low frequencies.}
    \label{fig:example_high_frequencies}
  \end{figure}
\end{example}

\begin{example} \label{ex:filter_1d_2}
  Let $L=1$ for simplicity. 
  Assume that 
  \[
  \hat \phi_k = \begin{cases}
    0 \, , & \text{if } |k| \neq 1 \, , \\ 
    4\pi  \, , & \text{if } |k| = 1 \, .
  \end{cases}
  \]
  Consider the initial data 
  \[
    u_0 = 0 \, , \quad v_0 = 2 \cos\big(2\pi x\big) = e^{2\pi i x} + e^{-2\pi i x} \, .
  \]
  Then
  \[
  \hat v_k(0) = \begin{cases}
    0 \, , & \text{if } |k| \neq 1 \, , \\ 
    1 \, , & \text{if } |k| = 1 \, .
  \end{cases}
  \]
  By~\eqref{eq:explicit_solution_1d_delta_zero}, the solution $u(t,x)$ satisfies
  \[
  \hat u_k(t) = \begin{cases}
    0 \, , & \text{if } |k| \neq 1 \, , \\
    t e^{-\frac{t}{2}} \, , & \text{if } |k| = 1 \, .
  \end{cases}
  \]
  We observe that the null space of the operator $P$ is given by
  \[
  N(P) = \Big\{v = \sum_{|k| \neq 1} \hat v_k e^{2\pi i k x} \in L^2(0,L) \Big\} \, .
  \]
  The orthogonal projection $Q = \Id - P$ on the null space of $P$ is simply the projection on the modes with $|k| \neq 1$.
  This implies that $u(t,\cdot) - Q[u(t,\cdot)]$ is given by the modes with $|k| = 1$.
  It follows that
  \[
 \|u(t,\cdot) - Q[u(t,\cdot)]\|^2_{H^1(0,1)} = \big(1+\big(\tfrac{2\pi}{L}\big)^2\big) | \hat u_1(t)|^2 + \big(1+(-2\pi)^2\big) |\hat u_{-1}(t)|^2 = 2 \big(1+(2\pi)^2\big) t^2 e^{-t} \, .
  \]
  Taking the square root, we obtain that the rate of convergence is not $e^{-t/2}$ as in Example~\ref{ex:filter_1d}, but it is still exponential $e^{-\gamma t}$ for $\gamma \in (0,\tfrac{1}{2})$.
\end{example}

The exponential rate of convergence is not the general behavior of the solution. 
It is strongly related to the structure of the frequency filter $P$, as we show in the next example, where the rate of convergence is subexponential.

\begin{example} \label{ex:filter_1d_3}
  In this example we show that the universal exponential decay is not always present.
  We provide an example of $P$ such that the following statement is not true: There exists $\gamma > 0$ and $M > 0$ such that for all initial data $u_0 \in H^1_0(0,1)$ and $v_0 \in L^2(0,1)$, the solution $u(t,\cdot)$ with initial data $u_0$ and $v_0$ satisfies 
  \[ \label{eq:universal_exponential_decay}
  \limsup_{t \to + \infty} \frac{\|u(t,\cdot)\|_{H^1(0,1)}}{e^{-\gamma t}( \|u_0\|_{H^1(0,1)} + \|v_0\|_{L^2(0,1)})} \leq M \, .
  \]
  Let $L=1$ for simplicity. 
  Let us fix $(\hat \phi_k)_{k \in \Z}$ that satisfies the following properties:
  \begin{itemize}
    \item $\hat \phi_k > 0$ for all $k \in \Z$;
    \item $\hat \phi_k^2 - ( 4 \pi k)^2 < 0$ for all $k \in \Z$, $k \neq 0$; 
    \item $\hat \phi_{k} = \hat \phi_{-k}$ for all $k \in \Z$;
    \item $\liminf_{|k| \to +\infty} \hat \phi_k = 0$.
  \end{itemize} 
  The first condition is required just to simplify the example. 
  Indeed, it implies that the null space of the operator $P$ is given by $N(P) = \{0\}$. Hence, in this example we do not need to consider the orthogonal projection $Q$, and we simply have to analyze the convergence of the solutions to the zero function.

  Our claim is the following. 
  
  {\itshape \underline{Claim}}: Let $\gamma > 0$.
  Let $M > 0$.
  There exist initial data $u_0 \in H^1_0(0,1)$ and $v_0 \in L^2(0,1)$ and a sequence of times $t_n \to + \infty$ such that the solution $u(t_n,\cdot)$ with initial data $u_0$ and $v_0$ satisfies\footnote{This is precisely the negation of~\eqref{eq:universal_exponential_decay}.}
  \[
  \|u(t_n,\cdot)\|_{H^1(0,1)} > M e^{-\gamma t_n} \big( \|u_0\|_{H^1(0,1)} + \|v_0\|_{L^2(0,1)} \big)  \quad \text{for all } n \in \N \, .
  \]

  We divide the proof of the claim in several steps.

  {\itshape Step 1: Fixing a suitable frequency $k_0$.} Let us fix $\gamma > 0$ and $M > 0$.  
  We let $k_0 \in \N \sm \{0\}$ (depending on $\gamma$) be such that 
  \[ \label{eq:condition_k_0}
  \hat \phi_{k_0} < \frac{\gamma}{2} \quad \text{and} \quad 2(1+(2\pi k_0)^2) > \frac{\gamma^2}{2} \, .
  \]

  {\itshape Step 2: Constructing the initial datum.}
  In this step we define the initial data $u_0$ and $v_0$.
  We let
  \[
  u_0(x) = 2 \sin(2\pi i k_0 x) \, .
  \]
  Note that $u_0 \in H^1_0(0,1)$ and 
  \[
  \hat u_k(0) = \begin{cases}
    1 \, , & \text{if } k = k_0 \, , \\[1em]
    -1 \, , & \text{if } k = -k_0 \, , \\[1em]
    0 \, , & \text{otherwise} \, .
  \end{cases}
  \]
  Moreover, we let $v_0 = - \frac{1}{2} P [u_0] \in L^2(0,1)$, so that 
  \[
  \hat v_k(0) = - \frac{\hat \phi_k}{2} \hat u_k(0) \implies \frac{\hat \phi_k}{2} \hat u_k(0) + \hat v_k(0) = 0 \, , \quad \text{for all } k \in \Z \, .
  \] 
  Note that 
  \[
  \|u_0\|_{H^1(0,1)}^2 = \sum_{k \in \Z} (1+(2\pi k)^2) |\hat u_k(0)|^2 = 2 (1 + (2\pi k_0)^2) \, , 
  \]
  and, by~\eqref{eq:condition_k_0}, 
  \[ \label{eq:estimating_v_0_with_u_0}
    \|v_0\|_{L^2(0,1)}^2 = \sum_{k \in \Z} |\hat v_k(0)|^2 = |\hat v_{k_0}(0)|^2 + |\hat v_{-k_0}(0)|^2 = \frac{\hat \phi_{k_0}^2}{2} < \frac{\gamma^2}{2} < 2 (1 + (2\pi k_0)^2) =  \|u_0\|_{H^1(0,1)}^2  \, .
  \]

  {\itshape Step 4: Computing the solution.} By~\eqref{eq:explicit_solution_1d_delta_negative}, the solution $u(t,x)$ to the problem~\eqref{eq:PDE 1d} with initial data $u_0$ and $v_0$ has Fourier coefficients given by
  \[
  \begin{split}
    \hat u_k(t) & = e^{- t  \frac{\hat \phi_k}{2}}  \cos \Big( \tfrac{t}{2}\sqrt{(4\pi k)^2 - \hat \phi_k^2 }  \Big) \hat u_k(0) \\
    & = \begin{cases}
      e^{- t  \frac{\hat \phi_k}{2}}  \cos \Big( \tfrac{t}{2}\sqrt{(4\pi k)^2 - \hat \phi_k^2 }  \Big)  \, , & \text{if } k = k_0 \, , \\[1em]
      - e^{- t  \frac{\hat \phi_k}{2}}  \cos \Big( \tfrac{t}{2}\sqrt{(4\pi k)^2 - \hat \phi_k^2 }  \Big)   \, , & \text{if } k = -k_0 \, , \\[1em]
      0 \, , & \text{otherwise} \, .
    \end{cases}
  \end{split}
  \]

  {\itshape Step 4: Estimating the $H^1$ norm of the solution.} 
  By~\eqref{eq:condition_k_0} and by~\eqref{eq:estimating_v_0_with_u_0}, we get
  \[ \label{eq:lower_bound_H1}
  \begin{split}
    \| u(t,\cdot) \|_{H^1}^2 & = \sum_{k \in \Z} \big(1 + (2\pi k)^2\big) |\hat u_k(t)|^2 = e^{-t \hat \phi_{k_0}} \cos^2 \Big( \tfrac{t}{2}\sqrt{(4\pi k_0)^2 - \hat \phi_{k_0}^2 }  \Big) 2 (1 + (2\pi k_0)^2\big) \\
    & \geq e^{-\gamma t/2} \cos^2 \Big( \tfrac{t}{2}\sqrt{(4\pi k_0)^2 - \hat \phi_{k_0}^2 }  \Big) 2 (1 + (2\pi k_0)^2\big) \\
    & \geq e^{\gamma t/2} e^{- \gamma t} \cos^2 \Big( \tfrac{t}{2}\sqrt{(4\pi k_0)^2 - \hat \phi_{k_0}^2 }  \Big) \|u_0\|_{H^1(0,1)}^2 \\
    & \geq e^{\gamma t/2} e^{-\gamma t} \cos^2 \Big( \tfrac{t}{2}\sqrt{(4\pi k_0)^2 - \hat \phi_{k_0}^2 }  \Big) \frac{1}{2} \big( \|u_0\|_{H^1(0,1)}^2 + \|v_0\|_{L^2(0,1)}^2 \big) \, .
  \end{split}
  \]

  {\itshape Step 5: Choice of sequence of times.} First of all, we construct a sequence of times $t_n \to +\infty$ such that
  \[ \label{eq:sequence_of_times}
  \cos^2 \big( \tfrac{t_n}{2}\sqrt{(4\pi k_0)^2 - \hat \phi_{k_0}^2 }  \big) \geq \tfrac{1}{2} \, , \quad \text{ for all } n \in \N \, . 
\]
For, it is enough to choose $t_n = \frac{2\pi n}{\sqrt{(4\pi k_0)^2 - \hat \phi_{k_0}^2 }}$.
By choosing $n \geq n_0$ with $n_0$ large enough, we can additionally ensure that 
\[ \label{eq:exploding exponential}
\frac{1}{4} e^{ \gamma t_n/2} > M \, , \quad \text{for all } n \geq n_0 \, .
\] 
Putting~\eqref{eq:sequence_of_times} and~\eqref{eq:exploding exponential} in~\eqref{eq:lower_bound_H1}, we obtain that 
\[
\begin{split}
  \| u(t_n,\cdot) \|_{H^1}^2 & \geq \frac{1}{2} e^{ \gamma t_n/2} e^{-\gamma t_n} \cos^2 \Big( \tfrac{t_n}{2}\sqrt{(4\pi k_0)^2 - \hat \phi_{k_0}^2 }  \Big)  \big( \|u_0\|_{H^1(0,1)}^2 + \|v_0\|_{L^2(0,1)}^2 \big) \\
   & \geq M e^{-\gamma t_n}\big( \|u_0\|_{H^1(0,1)}^2 + \|v_0\|_{L^2(0,1)}^2 \big) \, .
\end{split}
\]
This concludes the proof of the claim.
\end{example}

\section{Exponential decay for projected solutions}

In this section we show that, under suitable assumptions on $P$, the solution is split into two components: one that decays exponentially fast (the projected solution) and one that solves the undamped wave equation (the orthogonal component).

The precise assumptions on the bounded linear operator $P \colon L^2(\Omega) \to L^2(\Omega)$ are the following:
\begin{enumerate}[label=(A\arabic*)]
  \item \label{item:thmA2} $P$ commutes pointwise\footnote{We use this nomenclature to distinguish the assumption from strong commutation. See, \eg, \cite[VIII.5]{ReeSim80}}  with the Dirichlet Laplacian, \ie, for every $\varphi \in H^2(\Omega) \cap H^1_0(\Omega)$ we have that $P[\varphi] \in H^2(\Omega)  \cap H^1_0(\Omega)$ and $P[\Delta \varphi] = \Delta P[\varphi]$. 
  \item \label{item:thmA3} $P \colon L^2(\Omega) \to L^2(\Omega)$ is an orthogonal projection, \ie, $P^2 = P$ and $P$ is self-adjoint;
\end{enumerate}

\begin{example}
  An operator $P$ of the form of Example~\ref{ex:filter_1d} satisfies~\ref{item:thmA2}--\ref{item:thmA3}.
\end{example}

\begin{remark}
  Assume that $P$ satisfies~\ref{item:thmA2}--\ref{item:thmA3}. 
  Then it also satisfies~\ref{eq:P1}--\ref{eq:P2}, see Example~\ref{ex:projection}.
\end{remark}

\begin{remark}
  Assume that $P$ satisfies~\ref{item:thmA2}--\ref{item:thmA3}.
  Let us show that $P$ preserves $H^1_0(\Omega)$, \ie, $P(H^1_0(\Omega)) \subset H^1_0(\Omega)$.
  Let $u \in H^1_0(\Omega)$ and let us show that $P[u] \in H^1_0(\Omega)$.
  Let us fix an approximating sequence $u_j \in H^2(\Omega) \cap H^1_0(\Omega)$ such that $u_j \to u$ in $H^1_0(\Omega)$.
  By~\ref{item:thmA2}, we have that $P[u_j] \in H^2(\Omega) \cap H^1_0(\Omega)$.
  Moreover, $P[u_j] \to P[u]$ in $L^2(\Omega)$, as $P$ is bounded.
  Let us estimate $\sup_j \|\nabla P[u_j] \|_{L^2(\Omega)}^2$.
  We integrate by parts and we exploit~\ref{item:thmA2}--\ref{item:thmA3} to obtain that 
  \[
  \begin{split}
    \|\nabla P[u_j] \|_{L^2(\Omega)}^2 & = - \langle P[u_j](x) ,  \Delta P[u_j] \rangle_{L^2(\Omega)} = - \langle P[u_j] , P[\Delta u_j] \rangle_{L^2(\Omega)} = - \langle P^2[u_j] , \Delta u_j \rangle_{L^2(\Omega)} \\
    & = - \langle P[u_j] , \Delta u_j \rangle_{L^2(\Omega)} = \langle \nabla P[u_j], \nabla u_j \rangle_{L^2(\Omega)} \leq \|\nabla P[u_j]\|_{L^2(\Omega)} \|\nabla u_j\|_{L^2(\Omega)} \, ,
  \end{split}  
  \]
  from which we deduce that 
  \[
  \sup_j \|\nabla P[u_j] \|_{L^2(\Omega)} \leq \sup_j \|\nabla u_j \|_{L^2(\Omega)} < + \infty \, .
  \]
  It follows that $P[u_j] \rightharpoonup  P[u]$ weakly in $H^1_0(\Omega)$, proving the claim.
\end{remark}

To the aim of the splitting result, we start with a preliminary result. 
The proof is classical, as it relies on the linearity of the equation. 
We provide a proof for the sake of completeness.

\begin{theorem} \label{thm:orthogonal_projection}
  Assume that $P$ satisfies~\ref{item:thmA2}--\ref{item:thmA3}. 
  Let $u_0 \in H^1_0(\Omega)$ and $v_0 \in L^2(\Omega)$.
  Let $(u(t), v(t))$ be the unique solution to~\eqref{eq:PDE} with initial datum $(u_0, v_0)$ provided by Theorem~\ref{thm:well-posedness}.
  Then $(P[u(t)], P[v(t)])$ is the unique solution (in the sense of Theorem~\ref{thm:well-posedness}) to the problem:
  \[ \label{eq:damped_wave_projection}
  \begin{cases}
    \de_{tt} w - \Delta w + \de_t w = 0 \, , & (t,x) \in (0,+ \infty) \times \Omega \, , \\
    w = 0 \, , & (t,x) \in (0,+ \infty) \times \de \Omega \, , \\
    w(0,x) = P[u_0](x) \, , \quad \de_t w(0,x) = P[v_0](x) \, , & x \in \Omega \, .
  \end{cases}
  \]
\end{theorem}
Moreover, there exist constants $M > 0$ and $\gamma > 0$ such that
\[ \label{eq:exponential_decay_projection}
\|P[u(t)]\|_{H^1(\Omega)}^2 + \|P[v(t)]\|_{L^2(\Omega)}^2 \leq M e^{-\gamma t} \, , \quad \text{for all } t \geq 0 \, .
\]
\begin{proof}
  We split the proof in several steps. 

  \emph{Step 1:} First of all we show that $t \mapsto u(t,\cdot) \in H^1_0(\Omega)$ is also a distributional solution to~\eqref{eq:PDE}, \ie, we have that  
  \[ \label{eq:distributional_solution}
  \begin{split}
    & \int_0^{+\infty} \int_\Omega u(t,x) \Big( \de_{tt} \varphi(t,x) -  \Delta \varphi(t,x) - \de_t P[\varphi(t,\cdot)](x) \Big) \, \d x \, \d t \\
    & \quad = - \int_\Omega u_0(x)  \de_t \varphi(0,x) \, \d x + \int_\Omega \Big(  u_0(x) P[\varphi(0,\cdot)](x) + v_0(x) \varphi(0,x) \Big) \, \d x  \, ,
  \end{split}
  \]
  for all $\varphi \in C^\infty_c(\R \times \Omega)$. 

  To see this, let us first work in the case $U_0 \in \Dom(\A)$, using the notation of Section~\ref{sec:well-posedness}.
  Then $U(t) = (u(t), v(t)) = S(t)U_0$ belongs to $C^1([0,+\infty);H) \cap C([0,+\infty); \Dom(\A))$ and is the unique solution to~\eqref{eq:ACP}. 
  We fix a curve $t \mapsto \Phi(t) \in H$ given by $\Phi(t) = (0, \varphi(t,\cdot))$ with $\varphi \in C^\infty_c(\R \times \Omega)$.
  We have that 
  \[
  \begin{split}
    - \langle U_0, \Phi(0) \rangle_H & = \int_0^{+\infty} \frac{\d}{\d t} \langle U(t), \Phi(t) \rangle_H \, \d t \\
    & = - \int_0^{+\infty}   \langle \A U(t), \Phi(t) \rangle_H \, \d t  + \int_0^{+\infty} \Big\langle U(t), \frac{\d}{\d t} \Phi(t) \Big\rangle_H   \, \d t \, ,
  \end{split}
  \]
  which reads, using the self-adjointness of $P$,
  \[
  \begin{split}
 - \langle v_0, \varphi(0) \rangle_{L^2(\Omega)} & = \int_0^{+\infty}   \langle \Delta u(t) - P[v(t)], \varphi(t) \rangle_{L^2(\Omega)}   \, \d t - \int_0^{+\infty}    \langle v(t), \de_t \varphi(t) \rangle_{L^2(\Omega)}   \, \d t \\
 & = \int_0^{+\infty}  \Big( \langle  u(t), \Delta \varphi(t) \rangle_{L^2(\Omega)}   -  \langle v(t), P[\varphi(t)] \rangle_{L^2(\Omega)} \Big) \, \d t \\
 & \quad + \int_0^{+\infty}    \langle v(t), \de_t \varphi(t) \rangle_{L^2(\Omega)}   \, \d t  \\
 & = \int_0^{+\infty}  \Big( \langle  u(t), \Delta \varphi(t) \rangle_{L^2(\Omega)}   -  \Big\langle \frac{\d}{\d t}u(t), P[\varphi(t)] \Big\rangle_{L^2(\Omega)} \Big) \, \d t \\
 & \quad + \int_0^{+\infty}    \Big\langle  \frac{\d}{\d t}u(t) , \de_t \varphi(t) \Big\rangle_{L^2(\Omega)}   \, \d t \, .
  \end{split}
  \]
  Then we substitute in the previous equation the following two identities:
  \[
  \begin{split}
    - \langle u_0, \de_t \varphi(0) \rangle_{L^2(\Omega)} & = \int_0^{+\infty }\frac{\d}{\d t} \langle u(t), \de_t \varphi(t) \rangle_{L^2(\Omega)} \, \d t \\
    & = \int_0^{+\infty } \Big\langle \frac{\d}{\d t} u(t), \de_t \varphi(t) \Big\rangle_{L^2(\Omega)} \, \d t + \int_0^{+\infty }  \langle u(t), \de_{tt} \varphi(t) \rangle_{L^2(\Omega)} \, \d t \, ,
  \end{split}
  \]
  and, using Remark~\ref{rem:P_commutes_with_det},
  \[
  \begin{split}
  - \langle u_0, P[\varphi(0)] \rangle_{L^2(\Omega)} & = \int_0^{+\infty } \frac{\d}{\d t} \langle u(t), P[\varphi(t)]\rangle_{L^2(\Omega)} \, \d t \\
  & = \int_0^{+\infty } \Big\langle \frac{\d}{\d t} u(t), P[\varphi(t)]\Big\rangle_{L^2(\Omega)} \, \d t + \int_0^{+\infty } \Big\langle u(t), \frac{\d}{\d t} P[\varphi(t)]\Big\rangle_{L^2(\Omega)} \, \d t \\
  & = \int_0^{+\infty } \Big\langle \frac{\d}{\d t} u(t), P[\varphi(t)] \Big\rangle_{L^2(\Omega)} \, \d t + \int_0^{+\infty } \langle u(t), P[\de_t \varphi(t)]\rangle_{L^2(\Omega)} \, \d t  \, ,
  \end{split}
  \]
  to obtain that
  \[
  \begin{split}
  & \langle u_0, \de_t \varphi(0) \rangle_{L^2(\Omega)}  - \langle u_0, P[\varphi(0)] \rangle_{L^2(\Omega)} - \langle v_0, \varphi(0) \rangle_{L^2(\Omega)} \\
  & \quad = \int_0^{+\infty}  \Big( \langle  u(t), \Delta \varphi(t) \rangle_{L^2(\Omega)} + \langle  u(t), P[\de_t \varphi(t)] \rangle_{L^2(\Omega)} \Big) \, \d t - \int_0^{+\infty}  \langle  u(t) , \de_{tt} \varphi(t) \rangle_{L^2(\Omega)}   \, \d t \, ,
  \end{split}
  \]
  which is precisely~\eqref{eq:distributional_solution}.

  If $U_0 \in H$ (not necessarily in $\Dom(\A)$), then~\eqref{eq:distributional_solution} is obtained by approximating $U_0$ in the $H$-norm with a sequence in $\Dom(\A)$ and then passing to the limit.

  \emph{Step 2}: In the condition~\eqref{eq:distributional_solution} it is enough to test the equation with $\varphi(t,x) = \zeta(t) \psi(x)$, where $\zeta \in C^\infty_c(\R)$ and $\psi \in C^\infty_c(\Omega)$.
  Hence, it reads 
  \[ \label{eq:distributional_solution_2}
  \begin{split}
    & \int_0^{+\infty} \int_\Omega u(t,x) \Big( \de_{tt} \zeta(t)\psi(x) -  \zeta(t) \Delta \psi(x) - \de_t \zeta(t) P[\psi](x) \Big) \, \d x \, \d t \\
    & \quad = - \int_\Omega u_0(x)  \de_t\zeta(0) \psi(x) \, \d x + \int_\Omega \Big( u_0(x) \zeta(0)  P[\psi](x)+ v_0(x) \zeta(0)\psi(x) \Big)  \, \d x  \, ,
  \end{split}
  \]
  for all $\zeta \in C^\infty_c(\R)$ and $\psi \in C^\infty_c(\Omega)$.

  \emph{Step 3}: Note that, by an approximation argument, \eqref{eq:distributional_solution_2} can be tested with $\psi \in H^2(\Omega) \cap H^1_0(\Omega)$.

  \emph{Step 4}: Given $\psi \in H^2(\Omega) \cap H^1_0(\Omega)$, by~\ref{item:thmA2} we have that $P[\psi] \in H^2(\Omega) \cap H^1_0(\Omega)$.
  Hence, we can use $P[\psi]$ in~\eqref{eq:distributional_solution_2} instead of $\psi$ to obtain that 
  \[ \label{eq:distributional_solution_3}
  \begin{split}
    & \int_0^{+\infty} \int_\Omega u(t,x) \Big( \de_{tt} \zeta(t)P[\psi](x) -  \zeta(t) \Delta P[\psi](x) - \de_t \zeta(t) P[P[\psi]](x) \Big) \, \d x \, \d t \\
    & \quad = - \int_\Omega u_0(x)  \de_t\zeta(0) P[\psi](x)  \, \d x + \int_\Omega \Big( u_0(x) \zeta(0) P[P[\psi]](x) + v_0(x) \zeta(0) P[\psi](x) \Big)\, \d x  \, .
  \end{split}
  \]
  Using the properties $P[\Delta \psi] = \Delta P [\psi]$ from~\ref{item:thmA2} and $P^2 = P$ from~\ref{item:thmA3}, we obtain that 
  \[ \label{eq:distributional_solution_4}
  \begin{split}
    & \int_0^{+\infty} \int_\Omega u(t,x) \Big( \de_{tt} \zeta(t)P[\psi](x) -  \zeta(t) P[\Delta \psi](x) - \de_t \zeta(t) P[\psi](x) \Big) \, \d x \, \d t \\
    & \quad = - \int_\Omega u_0(x) \de_t\zeta(0) P[\psi](x)  \, \d x + \int_\Omega \big( u_0(x)  + v_0(x) \big)  \zeta(0) P[\psi](x)\, \d x  \, .
  \end{split}
  \]
  Finally, since $P$ is self-adjoint, we have that
  \[ \label{eq:distributional_solution_5}
  \begin{split}
    & \int_0^{+\infty} \int_\Omega P[u(t,\cdot)](x) \Big( \de_{tt} \zeta(t)\psi(x) -  \zeta(t) \Delta \psi(x) - \de_t \zeta(t) \psi(x) \Big) \, \d x \, \d t \\
    & \quad = - \int_\Omega P[u_0](x) \de_t\zeta(0) \psi(x)  \, \d x + \int_\Omega \Big( P[u_0](x)  + P[v_0](x) \Big)  \zeta(0) \psi(x)\, \d x  \, .
  \end{split}
  \]
  
  \emph{Step 6}: Reasoning as for~\eqref{eq:distributional_solution_2}, we conclude that 
  \[ \label{eq:distributional_solution_6}
  \begin{split}
    & \int_0^{+\infty} \int_\Omega P[u(t,\cdot)](x) \Big( \de_{tt} \varphi(t,x) -  \Delta \varphi(t,x) - \de_t  \varphi(t,x) \Big) \, \d x \, \d t \\
    & \quad = - \int_\Omega P[u_0](x) \de_t\varphi(0,x)  \, \d x + \int_\Omega \Big( P[u_0](x)  + P[v_0](x) \Big)\varphi(0,x)\, \d x  \, .
  \end{split}
  \]
  for all $\varphi \in C^\infty_c(\R \times \Omega)$, which we recognize as the definition of distributional solution to~\eqref{eq:damped_wave_projection}.

  \emph{Step 7}: The distributional solution $w \in L^2((0,+\infty) \x \Omega)$ to~\eqref{eq:damped_wave_projection} is unique, hence \emph{a fortiori} must coincide with the solution provided by Theorem~\ref{thm:well-posedness}.
  This follows from a duality argument.
  To be precise, let us assume that $w \in L^2((0,+\infty) \x \Omega)$ is a distributional solution to~\eqref{eq:damped_wave} with initial data $P[u_0] = 0$ and $P[v_0] = 0$ and let us show that 
 \[  
  \begin{split}
    & \int_0^{+\infty} \int_\Omega w(t,x) \varphi(t,x)\, \d x \, \d t = 0 \, ,
  \end{split}
  \]
  for all $\varphi \in C^\infty_c(\R \times \Omega)$.
  Given $\varphi \in C^\infty_c(\R \times \Omega)$, let $T > 0$ be such that $\varphi(t,x) = 0$ for $t \geq T$.
  We consider a (strong) solution $\tilde \varphi$ to the final-time dual problem:
  \[ \label{eq:dual_damped_wave}
  \begin{cases}
    \de_{tt} \tilde \varphi - \Delta \tilde \varphi - \de_t \tilde \varphi = \varphi \, , & (t,x) \in (0,+ \infty) \times \Omega \, , \\
    \tilde \varphi = 0 \, , & (t,x) \in (0,+ \infty) \times \de \Omega \, , \\
    \tilde \varphi(T,x) = 0 \, , \quad \de_t \tilde \varphi(T,x) = 0 \, , & x \in \Omega \, .
  \end{cases}
  \]
  We approximate $\tilde \varphi$ with a sequence $\tilde \varphi_j \in C^\infty_c(\R \times \Omega)$ such that $\tilde \varphi_j \to \tilde \varphi$ in $H^2(\R \x \Omega)$. 
  Since $w$ is a distributional solution to~\eqref{eq:damped_wave_projection} with zero initial data, we have that
\[  
  \begin{split}
    & \int_0^{+\infty} \int_\Omega w(t,x) \Big( \de_{tt} \tilde \varphi_j(t,x) -  \Delta \tilde \varphi_j(t,x) - \de_t  \tilde \varphi_j(t,x) \Big) \, \d x \, \d t = 0 \, .
  \end{split}
  \]
  Passing to the limit as $j \to + \infty$, we obtain that
  \[  
  \begin{split}
    0 = & \int_0^{+\infty} \int_\Omega w(t,x) \Big( \de_{tt} \tilde \varphi(t,x) -  \Delta \tilde \varphi(t,x) - \de_t  \tilde \varphi(t,x) \Big) \, \d x \, \d t = \int_0^{+\infty} \int_\Omega w(t,x) \varphi(t,x)\, \d x \, \d t  \, ,
  \end{split}
  \]
  concluding the proof of the claim.
  
  \emph{Step 8}: The exponential decay follows from Remark~\ref{rem:exponential_decay_damped_wave}.
\end{proof}

\begin{theorem} \label{thm:exponential_decay}
  Assume that $P$ satisfies~\ref{item:thmA2}--\ref{item:thmA3}.
  Let $Q = \Id - P$.
  Let $u_0 \in H^1_0(\Omega)$ and $v_0 \in L^2(\Omega)$.
  Let $(u(t), v(t))$ be the unique solution to~\eqref{eq:PDE} with initial datum $(u_0, v_0)$ provided by Theorem~\ref{thm:well-posedness}.
  Then 
  \[ \label{eq:decomposition}
    (u(t), v(t)) = (Q[u(t)], Q[v(t)]) + (P[u(t)], P[v(t)])
  \]
  where $(Q[u(t)], Q[v(t)])$ is the unique solution to the undamped wave equation with initial datum $(Q[u_0], Q[v_0])$: 
  \[ \label{eq:undamped_wave}
  \begin{cases}
    \de_{tt} z - \Delta z = 0 \, , & (t,x) \in (0,+ \infty) \times \Omega \, , \\
    z = 0 \, , & (t,x) \in (0,+ \infty) \times \de \Omega \, , \\
    z(0,x) = Q[u_0](x) \, , \quad \de_t z(0,x) = Q[v_0](x) \, , & x \in \Omega \, ,
  \end{cases}
  \]
  and there exist constants $\gamma, M > 0$ such that
  \[
  \|u(t) - Q[u(t)]\|_{H^1_0(\Omega)} + \|v(t) - Q[v(t)]\|_{L^2(\Omega)} \leq M e^{-\gamma t} \, , \quad \text{for all } t \geq 0 \, .
  \]
\end{theorem}
\begin{proof}
  Note that $Q$ is the orthogonal projection on the null space of $P$. 
  The decomposition~\eqref{eq:decomposition} follows from the fact that $Q$ is the orthogonal projection on the orthogonal complement of the range of $P$.
  Moreover, $Q$ satisfies~\ref{item:thmA2}--\ref{item:thmA3} as well.

  As in the proof of Proposition~\ref{thm:orthogonal_projection}, one proves that $(Q[u(t)], Q[v(t)])$ is a distributional solution to the undamped wave equation~\eqref{eq:undamped_wave}. 
  Finally, the exponential decay follows from~\eqref{eq:exponential_decay_projection}.
\end{proof}

\begin{acknowledgements}
   The authors are members of Gruppo Nazionale per l'Analisi Matematica, la Probabilit\`a e le loro Applicazioni (GNAMPA) of the Istituto Nazionale di Alta Matematica (INdAM).
 
   They have been partially supported by the Research Project of National Relevance ``Evolution problems involving interacting scales'' granted by the Italian Ministry of Education, University and Research (MUR Prin 2022, project code 2022M9BKBC, Grant No. CUP D53D23005880006).
 
   They acknowledge financial support under the National Recovery and Resilience Plan (NRRP) funded by the European Union - NextGenerationEU -  
Project Title ``Mathematical Modeling of Biodiversity in the Mediterranean sea: from bacteria to predators, from meadows to currents'' - project code P202254HT8 - CUP B53D23027760001 -  
Grant Assignment Decree No. 1379 adopted on 01/09/2023 by the Italian Ministry of University and Research (MUR).

They were partially supported by the Italian Ministry of University and Research under the Programme ``Department of Excellence'' Legge 232/2016 (Grant No. CUP - D93C23000100001).

\end{acknowledgements}
 
\printbibliography

@book{Bro16,
  author    = {Brogliato, B.},
  title     = {Nonsmooth Mechanics},
  publisher = {Springer},
  year      = {2016}
}

@article{CocDNMadOrlZua24,
author = {Coclite, Giuseppe Maria and De Nitti, Nicola and Maddalena, Francesco and Orlando, Gianluca and Zuazua, Enrique},
title = {Exponential convergence to steady-states for trajectories of a damped dynamical system modeling adhesive strings},
journal = {Mathematical Models and Methods in Applied Sciences},
volume = {34},
number = {08},
pages = {1445-1482},
year = {2024},
doi = {10.1142/S021820252450026X}
}

@article{Har87,
    AUTHOR = {Haraux, A.},
    TITLE = {Semi-linear hyperbolic problems in bounded domains},
    JOURNAL = {Math. Rep.},
    FJOURNAL = {Mathematical Reports},
    VOLUME = {3},
    YEAR = {1987},
    NUMBER = {1},
    PAGES = {i--xxiv and 1--281},
    ISSN = {0275-7214},
    MRCLASS = {35L70 (34G20 35Bxx 47F05 47H15)},
    MRNUMBER = {1078761},
    MRREVIEWER = {M. Biroli},
}

@book{Ika97,
  author    = {Ikawa, M.},
  title     = {Hyperbolic Partial Differential Equations and Wave Phenomena},
  publisher = {American Mathematical Society (AMS)},
  year      = {1997}
}

@book{NosMor14,
  author    = {Nosonovsky, M. and Mortazavi, V.},
  title     = {Friction-Induced Vibrations and Self-Organization},
  publisher = {CRC Press},
  year      = {2014}
}

@book{NosBhu08,
  author    = {Nosonovsky, M. and Bhushan, B.},
  title     = {Multiscale Dissipative Mechanisms and Hierarchical Systems},
  publisher = {Springer},
  year      = {2008}
}

@book{ReeSim80,
  author    = {M. Reed and B. Simon},
  title     = {Methods of Modern Mathematical Physics, Volume 1: Functional Analysis},
  publisher = {Academic Press},
  year      = {1980},
  edition   = {Revised and Enlarged},
  address   = {New York},
  isbn      = {978-0125850506}
}

@article{ZhoSunLi18,
  author    = {F. Zhou and C. Sun and X. Li},
  title     = {Dynamics for the damped wave equations on time-dependent domains},
  journal   = {Discrete and Continuous Dynamical Systems-B},
  volume    = {23},
  number    = {4},
  pages     = {1645--1674},
  year      = {2018}
}

@article{LazMolRivSol22,
  author    = {G. Lazzaroni and R. Molinarolo and F. Riva and F. Solombrino},
  title     = {On the wave equation on moving domains: regularity, energy balance and application to dynamic debonding},
  journal   = {Interfaces and Free Boundaries},
  volume    = {25},
  number    = {3},
  pages     = {401--454},
  year      = {2022}
}

@article{BurKel78,
  author    = {R. Burridge and J. B. Keller},
  title     = {Peeling, slipping and cracking—some one-dimensional free-boundary problems in mechanics},
  journal   = {SIAM Review},
  volume    = {20},
  number    = {1},
  pages     = {31--61},
  year      = {1978}
}

@article{DMLazNar16,
  author    = {G. Dal Maso and G. Lazzaroni and L. Nardini},
  title     = {Existence and uniqueness of dynamic evolutions for a peeling test in dimension one},
  journal   = {Journal of Differential Equations},
  volume    = {261},
  number    = {9},
  pages     = {4897--4923},
  year      = {2016}
}

@article{CooBar73,
  author    = {J. Cooper and C. Bardos},
  title     = {A nonlinear wave equation in a time dependent domain},
  journal   = {Journal of Mathematical Analysis and Applications},
  volume    = {42},
  number    = {1},
  pages     = {29--60},
  year      = {1973}
}

@article{LazMolSol22,
  author    = {G. Lazzaroni and R. Molinarolo and F. Solombrino},
  title     = {Radial solutions for a dynamic debonding model in dimension two},
  journal   = {Nonlinear Analysis},
  volume    = {219},
  pages     = {112822},
  year      = {2022}
}

@article{CocDevMad21,
  author    = {G. Coclite and G. Devillanova and F. Maddalena},
  title     = {Waves in Flexural Beams with Nonlinear Adhesive Interaction},
  journal   = {Milan Journal of Mathematics},
  volume    = {89},
  pages     = {329--344},
  year      = {2021}
}

@article{CocFloLigMad17,
  author    = {G. M. Coclite and G. Florio and M. Ligabo and F. Maddalena},
  title     = {Nonlinear waves in adhesive strings},
  journal   = {SIAM Journal on Applied Mathematics},
  volume    = {77},
  number    = {2},
  pages     = {347--360},
  year      = {2017}
}

@article{MadPer08,
  author    = {F. Maddalena and D. Percivale},
  title     = {Variational models for peeling problems},
  journal   = {Interfaces and Free Boundaries},
  volume    = {10},
  number    = {4},
  pages     = {503--516},
  year      = {2008}
}

\end{document}